\RequirePackage{fix-cm}
\documentclass[smallextended]{svjour3}       
\smartqed  
\usepackage{amssymb,graphicx,amsmath,bm}
\DeclareMathOperator{\supp}{supp}

\begin{document}

\title{Controlling conditional expectations by zero-determinant strategies}


\author{Masahiko Ueda}


\institute{Masahiko Ueda \at
              Graduate School of Sciences and Technology for Innovation, Yamaguchi University, Yamaguchi 753-8511, Japan \\
              \email{m.ueda@yamaguchi-u.ac.jp}           
}

\date{Received: date / Accepted: date}

\maketitle

\begin{abstract}
Zero-determinant strategies are memory-one strategies in repeated games which unilaterally enforce linear relations between expected payoffs of players.
Recently, the concept of zero-determinant strategies was extended to the class of memory-$n$ strategies with $n\geq 1$, which enables more complicated control of payoffs by one player.
However, what we can do by memory-$n$ zero-determinant strategies is still not clear.
Here, we show that memory-$n$ zero-determinant strategies in repeated games can be used to control conditional expectations of payoffs.
Equivalently, they can be used to control expected payoffs in biased ensembles, where a history of action profiles with large value of bias function is more weighted.
Controlling conditional expectations of payoffs is useful for strengthening zero-determinant strategies, because players can choose conditions in such a way that only unfavorable action profiles to one player are contained in the conditions. 
We provide several examples of memory-$n$ zero-determinant strategies in the repeated prisoner's dilemma game.
We also explain that a deformed version of zero-determinant strategies is easily extended to the memory-$n$ case.
\end{abstract}

\section{Introduction}
\label{sec:intro}
Repeated games have succeeded in explaining cooperative behavior in the prisoner's dilemma situation, where defection is more favorable than cooperation \cite{FudTir1991,OsbRub1994}.
Recently, finite-memory strategies (strategies with finite recall) in repeated games have attracted much attention in game theory, because the rationality of real agents is bounded \cite{Rub1998}.
In computer science, agents with bounded rationality were modeled by finite automata, and equilibria of such agents have been investigated \cite{Ney1985,Rub1986,KalSta1988,NeyOka1999}.
In evolutionary biology, evolutionary stability of finite-memory strategies has been mainly focused on \cite{NowSig1992,NowSig1993,IFN2005}.
The class of memory-one strategies contains several representative strategies in the repeated prisoner's dilemma game, such as the Grim Trigger strategy \cite{Fri1971}, the Tit-for-Tat strategy \cite{RCO1965,AxeHam1981}, and the Win-Stay Lose-Shift strategy \cite{NowSig1993}.
Moreover, longer-memory strategies have recently been investigated since longer memory enables agents more complicated behavior \cite{LiKen2013,YBC2017,HMCN2017,MurBae2018,MurBae2020}.

In 2012, two physicists, William Press and Freeman Dyson, discovered a novel class of memory-one strategies, called zero-determinant (ZD) strategies, in the infinitely repeated prisoner's dilemma game \cite{PreDys2012}.
Counterintuitively, ZD strategies unilaterally control expected payoffs of players by enforcing linear relations between expected payoffs.
Since the discovery of ZD strategies, many extensions have been done, including extensions to multi-player multi-action stage games \cite{HWTN2014,PHRT2015,Guo2014,McAHau2016,HDND2016}, extensions to games with imperfect monitoring \cite{HRZ2015,MamIch2019,UedTan2020}, extensions to games with a discounting factor \cite{HTS2015,McAHau2016,IchMas2018,MamIch2020}, extensions to asynchronous games \cite{McAHau2017,You2017}, an extension to linear relations between moments of payoffs \cite{Ued2021}, and an extension to long-memory strategies \cite{Ued2021b}.
In addition, evolutionary stability of ZD strategies, such as extortionate ZD strategy and generous ZD strategy, in the repeated prisoner's dilemma game has been substantially investigated \cite{StePlo2012,HNS2013,AdaHin2013,StePlo2013,HNT2013,SzoPer2014}.
Human experiments also compared performance of extortionate ZD strategy and that of generous ZD strategy \cite{HRM2016,WZLZX2016}.
Furthermore, mathematical properties of the situation where several players take ZD strategies were investigated \cite{HDND2016,UedTan2020}.

In this paper, we provide an interpretation about the ability of memory-$n$ ZD strategies \cite{Ued2021b}.
Although memory-$n$ ZD strategies were originally introduced as strategies which unilaterally enforce linear relations between correlation functions of payoffs, we here elucidate that the fundamental ability of memory-$n$ ZD strategies is that they unilaterally enforce linear relations between conditional expectations of payoffs.
Equivalently, we can rephrase that memory-$n$ ZD strategies unilaterally enforce linear relations between expected payoffs in biased ensembles \cite{BecSch1993,LAW2005,GKP2006,GJLet2007,JacSol2010,UedSas2015,NyaTou2017}.
The results in Ref. \cite{Ued2021b} can be derived from this interpretation.
We also provide examples of memory-$n$ ZD strategies in the repeated prisoner's dilemma game.
Since expected payoffs conditional on previous action profiles are used in linear relations, players can choose conditions in such a way that only action profiles unfavorable to one player are contained in the conditions, which may result in strengthening original memory-one ZD strategies.
Furthermore, we show that extension of deformed ZD strategies \cite{Ued2021} to the memory-$n$ case is straightforward.

This paper is organized as follows.
In section \ref{sec:model}, we introduce a model of repeated games.
In section \ref{sec:previous}, we review ZD strategies.
In section \ref{sec:trigger}, we show that there exist strategies which unilaterally enforce probability zero to specific action profiles.
In section \ref{sec:conditional}, we introduce the concept of biased memory-$n$ ZD strategies, and show that they unilaterally enforce linear relations between expected payoffs in biased ensembles or probability zero for a set of action profiles.
In this section, we also discuss that the factorable memory-$n$ ZD strategies in Ref. \cite{Ued2021b} can be derived from biased memory-$n$ ZD strategies.
In section \ref{sec:examples}, we provide examples of biased memory-$n$ ZD strategies in the repeated prisoner's dilemma game.
In section \ref{sec:mnDZDS}, we introduce memory-$n$ version of deformed ZD strategies and provide several examples.
Section \ref{sec:conclusion} is devoted to concluding remarks.

\section{Model}
\label{sec:model}
We consider a repeated game with $N$ players.
The set of players is described as $\mathcal{N}:=\left\{ 1, \cdots, N \right\}$.
The action of player $a\in \mathcal{N}$ in a one-shot game is written as $\sigma_a \in A_a := \left\{ 1, \cdots, M_a \right\}$, where $M_a<\infty$ is the number of action of player $a$.
We define $\mathcal{A}:=\prod_{a=1}^N A_a$.
We collectively write $\bm{\sigma}:=\left( \sigma_1, \cdots,  \sigma_N \right)\in \mathcal{A}$ and call $\bm{\sigma}$ an action profile.
The payoff of player $a$ when the action profile is $\bm{\sigma}$ is described as $s_a\left( \bm{\sigma} \right)$.
We also write a probability $M$-simplex by $\Delta_M$.
We consider the situation that the game is repeated infinitely.
We write an action of player $a$ in $t$-th round $(t\geq 1)$ by $\sigma_a(t)$.
The (behavior) strategy of player $a$ is described as $\mathcal{T}_a := \left\{ T^{(t)}_a \right\}_{t=1}^\infty$, where $T^{(t)}_a: \mathcal{A}^{t-1} \to \Delta_{M_a}$ is the conditional probability at $t$-th round.
We write the expectation of the quantity $B$ with respect to strategies of all players by $\mathbb{E}[B]$.
We introduce a discounting factor by $\delta$, which satisfies $0\leq \delta \leq 1$.
The payoff of player $a$ in the repeated game is defined by
\begin{eqnarray}
 \mathcal{S}_a &:=& (1-\delta) \mathbb{E} \left[ \sum_{t=1}^\infty \delta^{t-1} s_a\left( \bm{\sigma}(t) \right) \right]
\end{eqnarray}
for $0\leq \delta < 1$, and
\begin{eqnarray}
 \mathcal{S}_a &:=& \mathbb{E} \left[ \lim_{T\rightarrow \infty} \frac{1}{T} \sum_{t=1}^T s_a\left( \bm{\sigma}(t) \right) \right]
\end{eqnarray}
for $\delta=1$.
In this paper, we consider only the case $\delta=1$.
Below we write $\sum_{\bm{\sigma}\in \mathcal{A}}$ and $\sum_{\sigma_a \in A_a} (\forall a)$ as $\sum_{\bm{\sigma}}$ and $\sum_{\sigma_a}$, respectively.

The payoff is rewritten as
\begin{eqnarray}
 \mathcal{S}_a &=& \lim_{T\rightarrow \infty} \frac{1}{T} \sum_{t=1}^T \sum_{\bm{\sigma}(t)} \cdots \sum_{\bm{\sigma}(1)} s_a\left( \bm{\sigma}(t) \right) \mathbb{P}_t\left( \bm{\sigma}(t), \cdots, \bm{\sigma}(1) \right),
\end{eqnarray}
where we have defined the joint probability distribution of action profiles $\left\{ \bm{\sigma}(t^\prime) \right\}_{t^\prime=1}^{t}$
\begin{eqnarray}
 \mathbb{P}_t\left( \bm{\sigma}(t), \cdots, \bm{\sigma}(1) \right) &:=& \prod_{t^\prime=1}^t \prod_{a=1}^N T^{(t^\prime)}_a \left( \sigma_a(t^\prime) | \bm{\sigma}(t^\prime-1), \cdots, \bm{\sigma}(1) \right). \nonumber \\
 &&
\end{eqnarray}
It should be noted that $\mathbb{P}_t$ satisfies the recursion relation
\begin{eqnarray}
 \mathbb{P}_{t+1}\left( \bm{\sigma}(t+1), \cdots, \bm{\sigma}(1) \right) &=& \left( \prod_{a=1}^N T^{(t+1)}_a \left( \sigma_a(t+1) | \bm{\sigma}(t), \cdots, \bm{\sigma}(1) \right) \right) \mathbb{P}_t\left( \bm{\sigma}(t), \cdots, \bm{\sigma}(1) \right). \nonumber \\
 && 
 \label{eq:recursion}
\end{eqnarray}

We first introduce (time-independent) memory-$n$ strategies $(n\geq 0)$.
\begin{definition}
\label{def:TImn}
A strategy of player $a$ is a \emph{(time-independent) memory-$n$ strategy} $(n\geq 0)$ when it is written in the form
\begin{eqnarray}
 T_a^{(t)}\left( \sigma_a(t) | \bm{\sigma}(t-1), \cdots, \bm{\sigma}(1) \right) &=& T_a\left( \sigma_a(t) | \bm{\sigma}(t-1), \cdots, \bm{\sigma}(t-n) \right) \nonumber \\
 && \quad  \left( \forall \sigma_a(t), \forall \left\{ \bm{\sigma}(t^\prime) \right\}_{t^\prime=1}^{t-1} \right)
\end{eqnarray}
for all $t> n$ with some common conditional probability $T_a$.
\end{definition}
Throughout this paper, we consider the situation that a player $\exists a\in \mathcal{N}$ uses a memory-$n$ strategy.
We remark that strategies of players $-a:=\mathcal{N}\backslash \{ a \}$ are arbitrary.
We define $\bm{\sigma}_{-a}:=\bm{\sigma}\backslash\sigma_a$.
For $t\geq n$, we also define the marginal probability distribution of the last $n$ action profiles obtained from $\mathbb{P}_t$ by
\begin{eqnarray}
 P_t\left( \bm{\sigma}(t), \cdots, \bm{\sigma}(t-n+1) \right) &:=& \sum_{\bm{\sigma}(t-n)} \cdots \sum_{\bm{\sigma}(1)} \mathbb{P}_t\left( \bm{\sigma}(t), \cdots, \bm{\sigma}(1) \right).
\end{eqnarray}
By taking summation of the both sides of Eq. (\ref{eq:recursion}) with respect to $\bm{\sigma}_{-a}(t+1)$, $\bm{\sigma}(t)$, $\cdots$, and $\bm{\sigma}(1)$ for $t\geq n$, the left-hand-side becomes
\begin{eqnarray}
&& \sum_{\bm{\sigma}_{-a}(t+1)\in\prod_{a^\prime \neq a}A_{a^\prime}} \sum_{\bm{\sigma}(t)} \cdots \sum_{\bm{\sigma}(1)} \mathbb{P}_{t+1}\left( \bm{\sigma}(t+1), \cdots, \bm{\sigma}(1) \right) \nonumber \\
&=& \sum_{\bm{\sigma}_{-a}(t+1)\in\prod_{a^\prime \neq a}A_{a^\prime}} \sum_{\bm{\sigma}(t)} \cdots \sum_{\bm{\sigma}(t-n+2)} P_{t+1}\left( \bm{\sigma}(t+1), \cdots, \bm{\sigma}(t-n+2) \right) \nonumber \\
&=& \sum_{\bm{\sigma}^{(-1)}} \cdots \sum_{\bm{\sigma}^{(-n)}} \delta_{\sigma_a^{(-1)}, \sigma_a(t+1)} P_{t+1} \left( \bm{\sigma}^{(-1)}, \cdots, \bm{\sigma}^{(-n)} \right),
\end{eqnarray}
where $\delta_{\sigma, \sigma^\prime}$ represents the Kronecker delta, which takes $1$ for $\sigma=\sigma^\prime$ and $0$ otherwise.
(The last line is obtained by renaming variables.)
The right-hand-side becomes
\begin{eqnarray}
&& \sum_{\bm{\sigma}_{-a}(t+1)\in\prod_{a^\prime \neq a}A_{a^\prime}} \sum_{\bm{\sigma}(t)} \cdots \sum_{\bm{\sigma}(1)} \left( \prod_{b=1}^N T^{(t+1)}_b \left( \sigma_b(t+1) | \bm{\sigma}(t), \cdots, \bm{\sigma}(1) \right) \right) \mathbb{P}_t\left( \bm{\sigma}(t), \cdots, \bm{\sigma}(1) \right) \nonumber \\
&=& \sum_{\bm{\sigma}(t)} \cdots \sum_{\bm{\sigma}(1)} T_a \left( \sigma_a(t+1) | \bm{\sigma}(t), \cdots, \bm{\sigma}(t-n+1) \right) \mathbb{P}_t\left( \bm{\sigma}(t), \cdots, \bm{\sigma}(1) \right) \nonumber \\
&=& \sum_{\bm{\sigma}(t)} \cdots \sum_{\bm{\sigma}(t-n+1)} T_a \left( \sigma_a(t+1) | \bm{\sigma}(t), \cdots, \bm{\sigma}(t-n+1) \right) P_t\left( \bm{\sigma}(t), \cdots, \bm{\sigma}(t-n+1) \right) \nonumber \\
&=& \sum_{\bm{\sigma}^{(-1)}} \cdots \sum_{\bm{\sigma}^{(-n)}} T_a\left( \sigma_a(t+1) | \bm{\sigma}^{(-1)}, \cdots, \bm{\sigma}^{(-n)} \right) P_t \left( \bm{\sigma}^{(-1)}, \cdots, \bm{\sigma}^{(-n)} \right).
\end{eqnarray}
By renaming $\sigma_a(t+1)\rightarrow \sigma_a$, we obtain
\begin{eqnarray}
 0 &=& \sum_{\bm{\sigma}^{(-1)}} \cdots \sum_{\bm{\sigma}^{(-n)}} T_a\left( \sigma_a | \bm{\sigma}^{(-1)}, \cdots, \bm{\sigma}^{(-n)} \right) P_t \left( \bm{\sigma}^{(-1)}, \cdots, \bm{\sigma}^{(-n)} \right) \nonumber \\
 && - \sum_{\bm{\sigma}^{(-1)}} \cdots \sum_{\bm{\sigma}^{(-n)}} \delta_{\sigma_a^{(-1)}, \sigma_a} P_{t+1} \left( \bm{\sigma}^{(-1)}, \cdots, \bm{\sigma}^{(-n)} \right) 
\end{eqnarray}
for $t\geq n$.
Then, by calculating $\lim_{T\rightarrow \infty} \frac{1}{T} \sum_{t=n}^{T+n-1}$ of both sides, we finally obtain
\begin{eqnarray}
 0 &=& \sum_{\bm{\sigma}^{(-1)}} \cdots \sum_{\bm{\sigma}^{(-n)}} \left[ T_a\left( \sigma_a | \bm{\sigma}^{(-1)}, \cdots, \bm{\sigma}^{(-n)} \right) - \delta_{\sigma_a^{(-1)}, \sigma_a} \right] P^* \left( \bm{\sigma}^{(-1)}, \cdots, \bm{\sigma}^{(-n)} \right), \nonumber \\
 &&
\end{eqnarray}
where we have introduced the limit distribution
\begin{eqnarray}
 P^* \left( \bm{\sigma}^{(-1)}, \cdots, \bm{\sigma}^{(-n)} \right) &:=& \lim_{T\rightarrow \infty} \frac{1}{T} \sum_{t=n}^{T+n-1} P_t \left( \bm{\sigma}^{(-1)}, \cdots, \bm{\sigma}^{(-n)} \right).
\end{eqnarray}
Therefore, we obtain the generalized version of Akin's lemma \cite{Aki2012,Ued2021b}:
\begin{lemma}
\label{lemma:Akin}
For the quantity
\begin{eqnarray}
 \hat{T}_a\left( \sigma_a | \bm{\sigma}^{(-1)}, \cdots, \bm{\sigma}^{(-n)} \right) &:=& T_a\left( \sigma_a | \bm{\sigma}^{(-1)}, \cdots, \bm{\sigma}^{(-n)} \right) -  \delta_{\sigma_a, \sigma^{(-1)}_a},
 \label{eq:PD}
\end{eqnarray}
the relation
\begin{eqnarray}
 0 &=& \sum_{\bm{\sigma}^{(-1)}} \cdots \sum_{\bm{\sigma}^{(-n)}} \hat{T}_a\left( \sigma_a | \bm{\sigma}^{(-1)}, \cdots, \bm{\sigma}^{(-n)} \right) P^{*} \left( \bm{\sigma}^{(-1)}, \cdots, \bm{\sigma}^{(-n)} \right)
 \label{eq:gAkin}
\end{eqnarray}
holds for arbitrary $\sigma_a$.
\end{lemma}
In other words, player $a$ unilaterally enforces linear relations between values of the limit probability distribution $P^{*}$ regardless of the strategies of other players.
Memory-one examples of such linear relations in the repeated prisoner's game is provided in Appendix \ref{app:memory-one}.
The quantity (\ref{eq:PD}) is called a Press-Dyson tensor (or a strategy tensor) \cite{Ued2021b}.

It should be noted that a Press-Dyson tensor $\hat{T}_a$ is solely controlled by player $a$.
Due to properties of a probability distribution $T_a$, a Press-Dyson tensor satisfies several relations.
First, it satisfies
\begin{eqnarray}
 \sum_{\sigma_a} \hat{T}_a \left( \sigma_a | \bm{\sigma}^{(-1)}, \cdots, \bm{\sigma}^{(-n)} \right) &=& 0
 \label{eq:PD_normalized}
\end{eqnarray}
for arbitrary $\left( \bm{\sigma}^{(-1)}, \cdots, \bm{\sigma}^{(-n)} \right)$ due to the normalization condition of $T_a$.
This implies that the number of linear relations (\ref{eq:gAkin}) enforced by player $a$ is at most $(M_a-1)$.
Second, it satisfies
\begin{eqnarray}
 \hat{T}_a \left( \sigma_a | \bm{\sigma}^{(-1)}, \cdots, \bm{\sigma}^{(-n)} \right) && \left\{
  \begin{array}{ll}
    \leq 0, & \left(\sigma_a = \sigma^{(-1)}_a \right) \\
    \geq 0, & \left(\sigma_a \neq \sigma^{(-1)}_a \right)
  \end{array}
  \right.
  \label{eq:property_strategy}
\end{eqnarray}
for all $\sigma_a$, $\bm{\sigma}^{(-1)}$, $\cdots$, $\bm{\sigma}^{(-n)}$.
Third, it satisfies
\begin{eqnarray}
 \left| \hat{T}_a \left( \sigma_a | \bm{\sigma}^{(-1)}, \cdots, \bm{\sigma}^{(-n)} \right) \right| &\leq& 1
 \label{eq:condition_strategy}
\end{eqnarray}
for all $\sigma_a$, $\bm{\sigma}^{(-1)}$, $\cdots$, $\bm{\sigma}^{(-n)}$.
The last two comes from the fact that $T_a$ takes value in $[0,1]$.

Below we write the expectation for the limit distribution $P^{*} \left( \bm{\sigma}^{(-1)}, \cdots, \bm{\sigma}^{(-n)} \right)$ by $\left\langle \cdots \right\rangle^{*}$, and note $s_0\left( \bm{\sigma} \right):=1$ $(\forall \bm{\sigma})$ for simplicity.
We remark that the payoff of player $\forall a^\prime \in \mathcal{N}$ is described as
\begin{eqnarray}
 \mathcal{S}_{a^\prime} &=& \lim_{T\rightarrow \infty} \frac{1}{T} \left( \sum_{t=1}^{n-1} \sum_{\bm{\sigma}(t)} \cdots \sum_{\bm{\sigma}(1)} s_{a^\prime}\left( \bm{\sigma}(t) \right) \mathbb{P}_t\left( \bm{\sigma}(t), \cdots, \bm{\sigma}(1) \right) \right. \nonumber \\
 && \left. + \sum_{t=n}^{T} \sum_{\bm{\sigma}(t)} \cdots \sum_{\bm{\sigma}(t-n+1)} s_{a^\prime}\left( \bm{\sigma}(t) \right) P_t\left( \bm{\sigma}(t), \cdots, \bm{\sigma}(t-n+1) \right)  \right) \nonumber \\
 &=& \sum_{\bm{\sigma}^{(-1)}} \cdots \sum_{\bm{\sigma}^{(-n)}} s_{a^\prime}\left( \bm{\sigma}^{(-1)} \right) P^*\left( \bm{\sigma}^{(-1)}, \cdots, \bm{\sigma}^{(-n)} \right) \nonumber \\
 &=& \left\langle s_{a^\prime} \left( \bm{\sigma}^{(-1)} \right) \right\rangle^{*}.
\end{eqnarray}
That is, the payoffs in the repeated game are calculated as expected payoffs in the limit distribution.
In the proof of Lemma \ref{lemma:Akin}, we have assumed that $P^*$ exists.
When $P^*$ does not exist, the payoffs in the repeated games cannot be defined.
Therefore, we consider only the case that $P^*$ exists.

\section{Previous studies}
\label{sec:previous}
Press and Dyson introduced the concept of zero-determinant strategies in repeated games \cite{PreDys2012}:
\begin{definition}
\label{def:moZDS}
A memory-one strategy of player $a$ is a \emph{zero-determinant (ZD) strategy} when its Press-Dyson vectors $\hat{T}_a$ can be written in the form
\begin{eqnarray}
 \sum_{\sigma_a} c_{\sigma_a} \hat{T}_a\left( \sigma_a | \bm{\sigma}^{(-1)} \right) &=& \sum_{b=0}^N \alpha_{b} s_{b} \left( \bm{\sigma}^{(-1)} \right) \quad (\forall \bm{\sigma}^{(-1)})
 \label{eq:moZDS}
\end{eqnarray}
with some nontrivial coefficients $\left\{ c_{\sigma_a} \right\}$ and $\left\{ \alpha_{b} \right\}$ (that is, not $c_{1}=\cdots=c_{M_a}=\mathrm{const.}$ and not $\alpha_0=\alpha_1=\cdots=\alpha_N=0$).
\end{definition}
(Press-Dyson tensors with $n=1$ are particularly called Press-Dyson vectors.)
Because Press-Dyson vectors satisfy Akin's lemma (Lemma \ref{lemma:Akin}), the following proposition holds:
\begin{proposition}[\cite{PreDys2012,McAHau2016}]
\label{prop:moZDS}
A ZD strategy (\ref{eq:moZDS}) unilaterally enforces a linear relation between expected payoffs:
\begin{eqnarray}
 0 &=& \sum_{b=0}^N \alpha_{b} \left\langle s_{b} \left( \bm{\sigma}^{(-1)} \right) \right\rangle^{*}.
\end{eqnarray}
\end{proposition}
That is, the expected payoffs can be unilaterally controlled by one ZD player.

Recently, a deformed version of ZD strategies was also introduced \cite{Ued2021}:
\begin{definition}
\label{def:deformedZDS}
A memory-one strategy of player $a$ is a \emph{deformed ZD strategy} when its Press-Dyson vectors $\hat{T}_a$ can be written in the form
\begin{eqnarray}
 \sum_{\sigma_a} c_{\sigma_a} \hat{T}_a\left( \sigma_a | \bm{\sigma}^{(-1)} \right) &=& \sum_{k_1=0}^\infty \cdots \sum_{k_N=0}^\infty \alpha_{k_1, \cdots, k_N} \prod_{b=1}^N s_{b} \left( \bm{\sigma}^{(-1)} \right)^{k_b} \quad \left( \forall \bm{\sigma}^{(-1)} \right) \nonumber \\
 &&
 \label{eq:deformedZDS}
\end{eqnarray}
with some nontrivial coefficients $\left\{ c_{\sigma_a} \right\}$ and $\left\{ \alpha_{k_1, \cdots, k_N} \right\}$.
\end{definition}
Due to the same reason as Proposition \ref{prop:moZDS}, the following proposition holds:
\begin{proposition}[\cite{Ued2021}]
\label{prop:deformedZDS}
A deformed ZD strategy (\ref{eq:deformedZDS}) unilaterally enforces a linear relation between moments of payoffs:
\begin{eqnarray}
 0 &=& \sum_{k_1=0}^\infty \cdots \sum_{k_N=0}^\infty \alpha_{k_1, \cdots, k_N} \left\langle \prod_{b=1}^N s_{b} \left( \bm{\sigma}^{(-1)} \right)^{k_b} \right\rangle^{*}.
\end{eqnarray}
\end{proposition}
That is, the moments of payoffs can also be unilaterally controlled by one ZD player.

Furthermore, Ueda extended the concept of ZD strategies to memory-$n$ strategies \cite{Ued2021b}:
\begin{definition}
\label{def:mnZDS}
A memory-$n$ strategy of player $a$ is a \emph{memory-$n$ ZD strategy} when its Press-Dyson tensors $\hat{T}_a$ can be written in the form
\begin{eqnarray}
 \sum_{\sigma_a} c_{\sigma_a} \hat{T}_a \left( \sigma_a | \bm{\sigma}^{(-1)}, \cdots, \bm{\sigma}^{(-n)} \right) &=& \sum_{b^{(-1)}=0}^N \cdots \sum_{b^{(-n)}=0}^N \alpha_{b^{(-1)},\cdots, b^{(-n)}} \prod_{m=1}^n s_{b^{(-m)}} \left( \bm{\sigma}^{(-m)} \right) \nonumber \\
 && \quad \left( \forall \left\{ \bm{\sigma}^{(-m)} \right\}_{m=1}^n \right)
 \label{eq:mnZDS}
\end{eqnarray}
with some nontrivial coefficients $\left\{ c_{\sigma_a} \right\}$ and $\left\{ \alpha_{b^{(-1)},\cdots, b^{(-n)}} \right\}$.
\end{definition}
Because of Lemma \ref{lemma:Akin}, the following proposition also holds:
\begin{proposition}[\cite{Ued2021b}]
\label{prop:mnZDS}
A memory-$n$ ZD strategy (\ref{eq:mnZDS}) unilaterally enforces a linear relation between correlation functions of payoffs:
\begin{eqnarray}
 0 &=& \sum_{b^{(-1)}=0}^N \cdots \sum_{b^{(-n)}=0}^N \alpha_{b^{(-1)},\cdots, b^{(-n)}} \left\langle \prod_{m=1}^n s_{b^{(-m)}} \left( \bm{\sigma}^{(-m)} \right) \right\rangle^{*}.
\end{eqnarray}
\end{proposition}

The purpose of this paper is reinterpreting memory-$n$ ZD strategies in terms of more elementary strategies.

\section{Probability-controlling strategies}
\label{sec:trigger}
We first prove that there exist memory-$n$ strategies which avoid an arbitrary action profile $\bm{\sigma}$ in the limit distribution.
We define $\delta_{\bm{\sigma}, \hat{\bm{\sigma}}}:=\prod_{a=1}^N \delta_{\sigma_a, \hat{\sigma}_a}$.

\begin{proposition}
\label{prop:trigger}
Given some action profiles $\left\{\hat{\bm{\sigma}}^{(-m)}\right\}_{m=1}^n$, memory-$n$ strategies of player $a$ of the form
\begin{eqnarray}
 \hat{T}_a \left( \sigma_a | \bm{\sigma}^{(-1)}, \cdots, \bm{\sigma}^{(-n)} \right) &=& \hat{T}_a^{(1)} \left( \sigma_a | \bm{\sigma}^{(-1)} \right) \prod_{m=1}^n \delta_{\bm{\sigma}^{(-m)}, \hat{\bm{\sigma}}^{(-m)}}, \nonumber \\
 &&  \quad \left( \forall \sigma_a,  \forall \left\{ \bm{\sigma}^{(-m)} \right\}_{m=1}^n \right)
 \label{eq:trigger}
\end{eqnarray}
where $\hat{T}_a^{(1)}$ is a Press-Dyson vector of a memory-one strategy satisfying $\hat{T}_a^{(1)} \left( \sigma_a^* | \hat{\bm{\sigma}}^{(-1)} \right) \neq 0$ for some $\sigma_a^*$, unilaterally enforce probability zero to the history $\hat{\bm{\sigma}}^{(-1)}, \cdots, \hat{\bm{\sigma}}^{(-n)}$:
\begin{eqnarray}
 P^{*} \left( \hat{\bm{\sigma}}^{(-1)}, \cdots, \hat{\bm{\sigma}}^{(-n)} \right) &=& 0.
 \label{eq:linear_trigger}
\end{eqnarray}
\end{proposition}

\begin{proof}
We consider Eq. (\ref{eq:trigger}) with $\sigma_a=\sigma_a^*$:
\begin{eqnarray}
 \hat{T}_a \left( \sigma_a^* | \bm{\sigma}^{(-1)}, \cdots, \bm{\sigma}^{(-n)} \right) &=& \hat{T}_a^{(1)} \left( \sigma_a^* | \bm{\sigma}^{(-1)} \right) \prod_{m=1}^n \delta_{\bm{\sigma}^{(-m)}, \hat{\bm{\sigma}}^{(-m)}}.
\end{eqnarray}
By calculating expectations of the both sides with respect to the limit distribution $P^{*} \left( \bm{\sigma}^{(-1)}, \cdots, \bm{\sigma}^{(-n)} \right)$ corresponding to the strategy (\ref{eq:trigger}), and by using Lemma \ref{lemma:Akin}, we obtain
\begin{eqnarray}
 0 &=& \sum_{\bm{\sigma}^{(-1)}, \cdots, \bm{\sigma}^{(-n)}} P^{*} \left( \bm{\sigma}^{(-1)}, \cdots, \bm{\sigma}^{(-n)} \right) \hat{T}_a^{(1)} \left( \sigma_a^* | \bm{\sigma}^{(-1)} \right) \prod_{m=1}^n \delta_{\bm{\sigma}^{(-m)}, \hat{\bm{\sigma}}^{(-m)}} \nonumber \\
 &=& P^{*} \left( \hat{\bm{\sigma}}^{(-1)}, \cdots, \hat{\bm{\sigma}}^{(-n)} \right) \hat{T}_a^{(1)} \left( \sigma_a^* | \hat{\bm{\sigma}}^{(-1)} \right).
\end{eqnarray}
By the assumption $\hat{T}_a^{(1)} \left( \sigma_a^* | \hat{\bm{\sigma}}^{(-1)} \right) \neq 0$, we obtain the equation (\ref{eq:linear_trigger}).
$\Box$
\end{proof}

Strategies of the form (\ref{eq:trigger}) can be used for avoiding some unfavorable situation $\left( \hat{\bm{\sigma}}^{(-1)}, \cdots, \hat{\bm{\sigma}}^{(-n)} \right)$.
We call strategies of the form (\ref{eq:trigger}) \emph{probability-controlling strategies}.
For example, the Grim Trigger strategy of player $1$ in the repeated prisoner's dilemma game can be regarded as a memory-one probability-controlling strategy avoiding the action profile (Cooperation, Defection), as we can see in Appendix \ref{app:memory-one}.
This fact provides another explanation about the property that the Grim Trigger strategy is unbeatable \cite{DOS2012b}.
We again discuss Grim Trigger in Section \ref{sec:examples}.

\section{Biased memory-$n$ ZD strategies}
\label{sec:conditional}
The limit distribution $P^{*}\left( \bm{\sigma}^{(-1)}, \cdots, \bm{\sigma}^{(-n)} \right)$ gives the joint probability of $n$ action profiles $\left( \bm{\sigma}^{(-1)}, \cdots, \bm{\sigma}^{(-n)} \right)$.
When we consider some real function $K \left( \bm{\sigma}^{(-1)}, \cdots, \bm{\sigma}^{(-n)} \right)$ and introduce the quantity
\begin{eqnarray}
 P_{K}\left( \bm{\sigma}^{(-1)}, \cdots, \bm{\sigma}^{(-n)} \right) &:=& \frac{P^{*}\left( \bm{\sigma}^{(-1)}, \cdots, \bm{\sigma}^{(-n)} \right) e^{K\left( \bm{\sigma}^{(-1)}, \cdots, \bm{\sigma}^{(-n)} \right)}}{\left\langle e^{K \left( \bm{\sigma}^{(-1)}, \cdots, \bm{\sigma}^{(-n)} \right)} \right\rangle^{*}},
\end{eqnarray}
this quantity can also be regarded as a probability distribution of $n$ action profiles $\left( \bm{\sigma}^{(-1)}, \cdots, \bm{\sigma}^{(-n)} \right)$.
In this ensemble of histories, a history with large $K$ is more weighted.
We call such ensemble $P_K$ as \emph{biased ensemble} biased by the function $K$.
Biased ensembles recently attract much attention in statistical mechanics of trajectories \cite{BecSch1993,LAW2005,GKP2006,GJLet2007,JacSol2010,UedSas2015,NyaTou2017}.

We now prove our main theorem.
\begin{theorem}
\label{th:biasZDS}
Let $\hat{T}_a^{(1)}$ be Press-Dyson vectors of a memory-one ZD strategy of player $a$:
\begin{eqnarray}
 \sum_{\sigma_a} c_{\sigma_a} \hat{T}_a^{(1)} \left( \sigma_a | \bm{\sigma}^{(-1)} \right) &=& \sum_{b=0}^N \alpha_b s_b \left( \bm{\sigma}^{(-1)} \right) \quad \left( \forall \bm{\sigma}^{(-1)} \right)
 \label{eq:moZDS_mod}
\end{eqnarray}
with some coefficients $\left\{ c_{\sigma_a}  \right\}$ and $\left\{ \alpha_b \right\}$.
Let $K: \mathcal{A}^n \to \mathbb{R}\cup \{ -\infty \}$ be a function satisfying $K \left( \cdot \right) < \infty$, and define 
\begin{eqnarray}
 K_{\mathrm{max}} &:=& \max_{\bm{\sigma}^{(-1)}, \cdots, \bm{\sigma}^{(-n)}} K\left( \bm{\sigma}^{(-1)}, \cdots, \bm{\sigma}^{(-n)} \right).
\end{eqnarray}
Then, a memory-$n$ strategy
\begin{eqnarray}
 \hat{T}_a \left( \sigma_a | \bm{\sigma}^{(-1)}, \cdots, \bm{\sigma}^{(-n)} \right) &=& \hat{T}_a^{(1)} \left( \sigma_a | \bm{\sigma}^{(-1)} \right) e^{K \left( \bm{\sigma}^{(-1)}, \cdots, \bm{\sigma}^{(-n)} \right) - K_{\mathrm{max}}} \nonumber \\
 && \quad \left( \forall \sigma_a,  \forall \left\{ \bm{\sigma}^{(-m)} \right\}_{m=1}^n \right)
 \label{eq:mnZDS_bias}
\end{eqnarray}
unilaterally enforces either a linear relation between expected payoffs in a biased ensemble (biased by the function $K$)
\begin{eqnarray}
 0 &=& \sum_{b=0}^N \alpha_b \frac{\left\langle s_b\left( \bm{\sigma}^{(-1)} \right)e^{K \left( \bm{\sigma}^{(-1)}, \cdots, \bm{\sigma}^{(-n)} \right)} \right\rangle^{*}}{\left\langle e^{K \left( \bm{\sigma}^{(-1)}, \cdots, \bm{\sigma}^{(-n)} \right)} \right\rangle^{*}}
 \label{eq:linear_biased}
\end{eqnarray}
or the relation
\begin{eqnarray}
 0 &=& P^{*} \left( \bm{\sigma}^{(-1)}, \cdots, \bm{\sigma}^{(-n)} \right) \quad \left( \forall \left( \bm{\sigma}^{(-1)}, \cdots, \bm{\sigma}^{(-n)} \right) \in \supp e^{K(\cdot)} \right),
 \label{eq:linear_biased_trivial}
\end{eqnarray}
where $\supp f$ represents the support of function $f$.
\end{theorem}

\begin{proof}
First, we check that tensors (\ref{eq:mnZDS_bias}) indeed satisfy the conditions of strategies, that is, Eqs. (\ref{eq:PD_normalized}), (\ref{eq:property_strategy}), and (\ref{eq:condition_strategy}).
Due to the equality
\begin{eqnarray}
 \sum_{\sigma_a}\hat{T}_a^{(1)} \left( \sigma_a | \bm{\sigma}^{(-1)} \right) &=& 0 \quad \left( \forall \bm{\sigma}^{(-1)} \right)
\end{eqnarray}
for Press-Dyson vectors of memory-one ZD strategies, we obtain
\begin{eqnarray}
 \sum_{\sigma_a} \hat{T}_a \left( \sigma_a | \bm{\sigma}^{(-1)}, \cdots, \bm{\sigma}^{(-n)} \right) &=& \sum_{\sigma_a} \hat{T}_a^{(1)} \left( \sigma_a | \bm{\sigma}^{(-1)} \right) e^{K \left( \bm{\sigma}^{(-1)}, \cdots, \bm{\sigma}^{(-n)} \right) - K_{\mathrm{max}}} \nonumber \\
 &=& 0
\end{eqnarray}
for arbitrary $\left\{ \bm{\sigma}^{(-m)} \right\}_{m=1}^n$, which implies Eq. (\ref{eq:PD_normalized}).
In addition, because the Press-Dyson vectors $\hat{T}_a^{(1)}$ of a memory-one ZD strategy satisfies
\begin{eqnarray}
 \hat{T}_a^{(1)} \left( \sigma_a | \bm{\sigma}^{(-1)} \right) && \left\{
  \begin{array}{ll}
    \leq 0, & \left(\sigma_a = \sigma^{(-1)}_a \right) \\
    \geq 0, & \left(\sigma_a \neq \sigma^{(-1)}_a \right)
  \end{array}
  \right.
\end{eqnarray}
for all $\sigma_a$ and $\bm{\sigma}^{(-1)}$, and the sign of $\hat{T}_a \left( \sigma_a | \bm{\sigma}^{(-1)}, \cdots, \bm{\sigma}^{(-n)} \right)$ is the same as that of $\hat{T}_a^{(1)} \left( \sigma_a | \bm{\sigma}^{(-1)} \right)$, we obtain Eq. (\ref{eq:property_strategy}) for all $\sigma_a$ and $\left\{ \bm{\sigma}^{(-m)} \right\}_{m=1}^n$
Furthermore, since the Press-Dyson vectors $\hat{T}_a^{(1)}$ of a memory-one ZD strategy satisfies
\begin{eqnarray}
 \left| \hat{T}_a^{(1)} \left( \sigma_a | \bm{\sigma}^{(-1)} \right) \right| &\leq& 1 \quad \left( \forall \sigma_a, \forall \bm{\sigma}^{(-1)} \right),
\end{eqnarray}
and then Eq. (\ref{eq:mnZDS_bias}) satisfies
\begin{eqnarray}
 \left| \hat{T}_a \left( \sigma_a | \bm{\sigma}^{(-1)}, \cdots, \bm{\sigma}^{(-n)} \right) \right| &=& \left| \hat{T}_a^{(1)} \left( \sigma_a | \bm{\sigma}^{(-1)} \right) e^{K \left( \bm{\sigma}^{(-1)}, \cdots, \bm{\sigma}^{(-n)} \right) - K_{\mathrm{max}}} \right| \nonumber \\
 &=& \left| \hat{T}_a^{(1)} \left( \sigma_a | \bm{\sigma}^{(-1)} \right) \right| e^{K \left( \bm{\sigma}^{(-1)}, \cdots, \bm{\sigma}^{(-n)} \right) - K_{\mathrm{max}}} \nonumber \\
 &\leq& \left| \hat{T}_a^{(1)} \left( \sigma_a | \bm{\sigma}^{(-1)} \right) \right| \nonumber \\
 &\leq& 1 \quad \left( \forall \sigma_a,  \forall \left\{ \bm{\sigma}^{(-m)} \right\}_{m=1}^n \right),
\end{eqnarray}
we obtain Eq. (\ref{eq:condition_strategy}) for all $\sigma_a$, $\bm{\sigma}^{(-1)}$, $\cdots$, $\bm{\sigma}^{(-n)}$

Next, from Eqs. (\ref{eq:mnZDS_bias}) and (\ref{eq:moZDS_mod}), we obtain
\begin{eqnarray}
 \sum_{\sigma_a} c_{\sigma_a} \hat{T}_a \left( \sigma_a | \bm{\sigma}^{(-1)}, \cdots, \bm{\sigma}^{(-n)} \right) &=& \sum_{b=0}^N \alpha_b s_b \left( \bm{\sigma}^{(-1)} \right) e^{K \left( \bm{\sigma}^{(-1)}, \cdots, \bm{\sigma}^{(-n)} \right) - K_{\mathrm{max}}}. \nonumber \\
 &&
\end{eqnarray}
By calculating expectations of the both sides with respect to the corresponding limit distribution $P^{*} \left( \bm{\sigma}^{(-1)}, \cdots, \bm{\sigma}^{(-n)} \right)$ and using Lemma \ref{lemma:Akin}, we obtain
\begin{eqnarray}
 0 &=& \sum_{b=0}^N \alpha_b \left\langle s_b\left( \bm{\sigma}^{(-1)} \right) e^{K \left( \bm{\sigma}^{(-1)}, \cdots, \bm{\sigma}^{(-n)} \right) - K_{\mathrm{max}}} \right\rangle^{*}.
 \label{eq:linear_biased_unnormalized}
\end{eqnarray}
Furthermore, if
\begin{eqnarray}
 \left\langle e^{K \left( \bm{\sigma}^{(-1)}, \cdots, \bm{\sigma}^{(-n)} \right)} \right\rangle^{*} &\neq& 0,
\end{eqnarray}
by dividing the both sides of Eq. (\ref{eq:linear_biased_unnormalized}) by $\left\langle e^{K \left( \bm{\sigma}^{(-2)}, \cdots, \bm{\sigma}^{(-n)} \right) - K_{\mathrm{max}}} \right\rangle^{*}$, we obtain Eq. (\ref{eq:linear_biased}).
Otherwise, the equality
\begin{eqnarray}
 \left\langle e^{K \left( \bm{\sigma}^{(-1)}, \cdots, \bm{\sigma}^{(-n)} \right)} \right\rangle^{*} &=& 0
\end{eqnarray}
holds.
This equality is rewritten as
\begin{eqnarray}
 0 &=& \sum_{\bm{\sigma}^{(-1)}, \cdots, \bm{\sigma}^{(-n)}} P^{*} \left( \bm{\sigma}^{(-1)}, \cdots, \bm{\sigma}^{(-n)} \right) e^{K \left( \bm{\sigma}^{(-1)}, \cdots, \bm{\sigma}^{(-n)} \right)}.
\end{eqnarray}
However, because $e^K$ is non-negative, this equality implies Eq. (\ref{eq:linear_biased_trivial}).
$\Box$
\end{proof}

Theorem \ref{th:biasZDS} can be regarded as an extension of Proposition \ref{prop:trigger}.
It should be noted that the limit probability distribution $P^{*} \left( \bm{\sigma}^{(-1)}, \cdots, \bm{\sigma}^{(-n)} \right)$ depends on strategies.
We call strategies in this Theorem as \emph{biased memory-$n$ ZD strategies}.
Controlling biased expectations of payoffs is useful for strengthening memory-one ZD strategies, because players can choose biased functions in such a way that unfavorable action profiles to one player are more weighted. 
Biased ensembles are used to amplify rare events in the same way as evolution in population genetics.
When we consider situation where each group with $N$ players is selected by fitness $e^K$, expected payoffs in such situation are calculated by our biased expectations. 
Such situation may be useful in the context of multilevel selection \cite{TraNow2006}, if $K$ is given by the total payoffs of all players in one group, for instance.
Furthermore, Theorem \ref{th:biasZDS} contains the following three corollaries.

\begin{corollary}
\label{cor:eleZDS}
Let $\hat{T}_a^{(1)}$ be Press-Dyson vectors of a memory-one ZD strategy satisfying Eq. (\ref{eq:moZDS_mod}).
Let $\left\{ \hat{\bm{\sigma}}^{(-m)} \right\}_{m=1}^n$ be some action profiles.
If $\sum_{b=0}^N \alpha_b s_b \left( \hat{\bm{\sigma}}^{(-1)} \right) \neq 0$, then a memory-$n$ strategy
\begin{eqnarray}
 \hat{T}_a \left( \sigma_a | \bm{\sigma}^{(-1)}, \cdots, \bm{\sigma}^{(-n)} \right) &=& \hat{T}_a^{(1)} \left( \sigma_a | \bm{\sigma}^{(-1)} \right) \prod_{m=1}^n \delta_{\bm{\sigma}^{(-m)}, \hat{\bm{\sigma}}^{(-m)}} \nonumber \\
 && \quad \left( \forall \sigma_a,  \forall \left\{ \bm{\sigma}^{(-m)} \right\}_{m=1}^n \right)
 \label{eq:mnZDS_ele}
\end{eqnarray}
unilaterally enforces the equation
\begin{eqnarray}
 P^{*} \left( \hat{\bm{\sigma}}^{(-1)}, \cdots, \hat{\bm{\sigma}}^{(-n)} \right) &=& 0.
 \label{eq:linear_ele}
\end{eqnarray}
\end{corollary}

\begin{proof}
By choosing the function $K$ such that
\begin{eqnarray}
 e^{K \left( \bm{\sigma}^{(-1)}, \cdots, \bm{\sigma}^{(-n)} \right)} &=& \prod_{m=1}^n \delta_{\bm{\sigma}^{(-m)}, \hat{\bm{\sigma}}^{(-m)}}
\end{eqnarray}
in Eq. (\ref{eq:mnZDS_bias}), we obtain
\begin{eqnarray}
 0 &=& \sum_{b=0}^N \alpha_b \left\langle s_b\left( \bm{\sigma}^{(-1)} \right) \prod_{m=1}^n \delta_{\bm{\sigma}^{(-m)}, \hat{\bm{\sigma}}^{(-m)}} \right\rangle^{*}.
\end{eqnarray}
(We remark that $K$ can be $-\infty$.)
It should be noted that $K_{\mathrm{max}}=0$.
By using the assumption $\sum_{b=0}^N \alpha_b s_b \left( \hat{\bm{\sigma}}^{(-1)} \right) \neq 0$, we obtain the result (\ref{eq:linear_ele}).
$\Box$
\end{proof}
Corollary \ref{cor:eleZDS} can also be derived directly from Proposition \ref{prop:trigger}.

\begin{corollary}
\label{cor:ZDS_conditional}
Let $\hat{T}_a^{(1)}$ be Press-Dyson vectors of a memory-one ZD strategy satisfying Eq. (\ref{eq:moZDS_mod}).
Let $\left\{ \hat{\bm{\sigma}}^{(-m)} \right\}_{m=2}^n$ be some action profiles.
Then a memory-$n$ strategy
\begin{eqnarray}
 \hat{T}_a \left( \sigma_a | \bm{\sigma}^{(-1)}, \cdots, \bm{\sigma}^{(-n)} \right) &=& \hat{T}_a^{(1)} \left( \sigma_a | \bm{\sigma}^{(-1)} \right) \prod_{m=2}^n \delta_{\bm{\sigma}^{(-m)}, \hat{\bm{\sigma}}^{(-m)}} \nonumber \\
 && \quad \left( \forall \sigma_a,  \forall \left\{ \bm{\sigma}^{(-m)} \right\}_{m=1}^n \right)
 \label{eq:mnZDS_conditional}
\end{eqnarray}
unilaterally enforces either a linear relation between conditional expectations of payoffs when the history of the previous $n-1$ action profiles is $\hat{\bm{\sigma}}^{(-2)}$, $\cdots$, $\hat{\bm{\sigma}}^{(-n)}$
\begin{eqnarray}
 0 &=& \sum_{b=0}^N \alpha_b \frac{\left\langle s_b\left( \bm{\sigma}^{(-1)} \right) \prod_{m=2}^n \delta_{\bm{\sigma}^{(-m)}, \hat{\bm{\sigma}}^{(-m)}} \right\rangle^{*}}{\left\langle \prod_{m=2}^n \delta_{\bm{\sigma}^{(-m)}, \hat{\bm{\sigma}}^{(-m)}} \right\rangle^{*}}
 \label{eq:linear_conditional}
\end{eqnarray}
or the relation
\begin{eqnarray}
 0 &=& \left\langle \prod_{m=2}^n \delta_{\bm{\sigma}^{(-m)}, \hat{\bm{\sigma}}^{(-m)}} \right\rangle^{*}.
 \label{eq:linear_conditional_trivial}
\end{eqnarray}
\end{corollary}

\begin{proof}
By choosing the function $K$ such that
\begin{eqnarray}
 e^{K \left( \bm{\sigma}^{(-1)}, \cdots, \bm{\sigma}^{(-n)} \right)} &=& \prod_{m=2}^n \delta_{\bm{\sigma}^{(-m)}, \hat{\bm{\sigma}}^{(-m)}}
\end{eqnarray}
in Eq. (\ref{eq:mnZDS_bias}), we obtain the relation (\ref{eq:linear_conditional}) or the relation (\ref{eq:linear_conditional_trivial}).
It should be noted that Eq. (\ref{eq:linear_conditional_trivial}) can be rewritten as
\begin{eqnarray}
 0 &=& \sum_{\bm{\sigma}^{(-1)}, \cdots, \bm{\sigma}^{(-n)}} P^{*} \left( \bm{\sigma}^{(-1)}, \cdots, \bm{\sigma}^{(-n)} \right) \prod_{m=2}^n \delta_{\bm{\sigma}^{(-m)}, \hat{\bm{\sigma}}^{(-m)}} \nonumber \\
 &=& \sum_{\bm{\sigma}^{(-1)}} P^{*} \left( \bm{\sigma}^{(-1)}, \hat{\bm{\sigma}}^{(-2)} \cdots, \hat{\bm{\sigma}}^{(-n)} \right),
\end{eqnarray}
which is a relation on the marginal distribution.
$\Box$
\end{proof}

We remark that memory-$n$ strategies (\ref{eq:mnZDS_conditional}) approach the strategy ``Repeat'' \cite{Aki2012}
\begin{eqnarray}
 \hat{T}_a \left( \sigma_a | \bm{\sigma}^{(-1)}, \cdots, \bm{\sigma}^{(-n)} \right) &=& 0 \quad \left( \forall \sigma_a,  \forall \left\{ \bm{\sigma}^{(-m)} \right\}_{m=1}^n \right)
\end{eqnarray}
(which repeats the previous action of the player) as $n$ increases, because $\prod_{m=2}^n \delta_{\bm{\sigma}^{(-m)}, \hat{\bm{\sigma}}^{(-m)}}=0$ for most $\left\{ \bm{\sigma}^{(-m)} \right\}_{m=1}^n$.
This property is used to control only conditional expectations when history of the action profiles is $\left( \hat{\bm{\sigma}}^{(-2)}, \cdots, \hat{\bm{\sigma}}^{(-n)} \right)$.

\begin{corollary}
\label{cor:mnZDS_factorized}
Memory-$n$ strategies of the form (\ref{eq:mnZDS_bias}) contain factorable memory-$n$ ZD strategies (\ref{eq:mnZDS}):
\begin{eqnarray}
 \sum_{\sigma_a} c_{\sigma_a} \hat{T}_a \left( \sigma_a | \bm{\sigma}^{(-1)}, \cdots, \bm{\sigma}^{(-n)} \right) &=& \prod_{m=1}^n \sum_{b_m=0}^N \alpha_{b_m}^{(m)} s_{b_m} \left( \bm{\sigma}^{(-m)} \right).
 \label{eq:mnZDS_factorized}
\end{eqnarray}
\end{corollary}

\begin{proof}
In Theorem \ref{th:biasZDS}, when the quantity $K\left( \bm{\sigma}^{(-1)}, \cdots, \bm{\sigma}^{(-n)} \right)$ is written in the form $\sum_{m=2}^n K_m \left( \bm{\sigma}^{(-m)} \right)$, and the quantity $G_m \left( \bm{\sigma} \right) := e^{ K_m \left( \bm{\sigma} \right)} \geq 0$ $(m\geq 2)$ is written by payoffs in the form
\begin{eqnarray}
 G_m \left( \bm{\sigma} \right) &=& \sum_{b_m=0}^N \alpha_{b_m}^{(m)} s_{b_m} \left( \bm{\sigma} \right)
\end{eqnarray}
with some coefficients $\left\{ \alpha_{b_m}^{(m)} \right\}$, then the strategies (\ref{eq:mnZDS_bias}) are reduced to memory-$n$ ZD strategies (\ref{eq:mnZDS_factorized}).
$\Box$
\end{proof}

\section{Examples: the prisoner's dilemma game}
\label{sec:examples}
\subsection{Setup}
\label{subsec:RPD}
We consider the repeated prisoner's dilemma game as an example.
The prisoner's dilemma game is the simplest two-player two-action game, where each player chooses cooperation (written as $1$) or defection (written as $2$) in each round \cite{PreDys2012}.
Payoffs of two players are 
\begin{eqnarray}
 \left( s_1 \left( 1, 1 \right), s_1 \left( 1, 2 \right), s_1 \left( 2, 1 \right), s_1 \left( 2, 2 \right) \right) &=& (R,S,T,P) \\
 \left( s_2 \left( 1, 1 \right), s_2 \left( 1, 2 \right), s_2 \left( 2, 1 \right), s_2 \left( 2, 2 \right) \right) &=& (R,T,S,P)
\end{eqnarray}
with $T>R>P>S$.
Although the Nash equilibrium of the one-shot game is $(2,2)$, it has been known that cooperative Nash equilibria exist when this game is repeated infinitely many times and the discounting factor is large enough (folk theorem).

\subsection{Biased memory-$n$ ZD strategies}
\label{subsec:mnZDS_RPD}
In the repeated prisoner's dilemma game, the following strategies are contained in biased memory-$n$ ZD strategies of player $1$:
\begin{eqnarray}
 \hat{T}_1\left( 1 | \bm{\sigma}^{(-1)}, \cdots, \bm{\sigma}^{(-n)} \right) &=& \prod_{m=1}^n G_m \left( \bm{\sigma}^{(-m)} \right)
\end{eqnarray}
where $G_1$ is a Press-Dyson vector of memory-one ZD strategies and $G_m$ $(m\geq 2)$ is non-negative quantity.
For example, when we assume that $2R>T+S$ and $2P<T+S$, they are
\begin{eqnarray}
 G_1 \left( \cdot \right) &\in& \left\{ -\frac{1}{T-P} \left[ s_2\left( \cdot \right) - P \right],  -\frac{1}{R-S} \left[ s_2\left( \cdot \right) - R \right], \frac{1}{T-S} \left[ s_1\left( \cdot \right) - s_2\left( \cdot \right) \right], \right. \nonumber \\
 && \left. -\frac{1}{B} \left[ s_1\left( \cdot \right) + s_2\left( \cdot \right) - (T+S) \right], \cdots \right\}
\end{eqnarray}
with $B:=\mathrm{max} \left\{ 2R-(T+S), (T+S)-2P \right\}$, and
\begin{eqnarray}
 G_m \left( \cdot \right) &\in& \left\{ \frac{1}{2(R-P)} \left[ s_1\left( \cdot \right) +  s_2\left( \cdot \right) -2P \right], \frac{1}{2(T-S)} \left[ s_1\left( \cdot \right) -  s_2\left( \cdot \right) + (T-S) \right], \right. \nonumber \\
 &&  -\frac{1}{2(T-S)} \left[ s_1\left( \cdot \right) -  s_2\left( \cdot\right) - (T-S) \right], \frac{1}{T-S} \left[ s_1\left( \cdot \right) - S \right], \nonumber \\
 && \left. - \frac{1}{T-S} \left[ s_1\left( \cdot \right) - T \right], \frac{1}{T-S} \left[ s_2\left( \cdot \right) - S \right], - \frac{1}{T-S} \left[ s_2\left( \cdot \right) - T \right], \cdots \right\}
\end{eqnarray}
for each $m\geq 2$.
We remark that the set of $G_m$ $(m\geq 1)$ also contains quantities which cannot be represented by payoffs.
A linear relation enforced by the ZD strategy is
\begin{eqnarray}
 0 &=& \left\langle \prod_{m=1}^n G_m \left( \bm{\sigma}^{(-m)} \right) \right\rangle^{*},
\end{eqnarray}
that is, a linear relation between correlation functions of payoffs.
These strategies contain all memory-two ZD strategies reported in Ref. \cite{Ued2021b}.

\subsection{Extension of Tit-for-Tat}
\label{subsec:ETFT}
For example, the memory-$n$ strategy
\begin{eqnarray}
 && \hat{T}_1\left( 1 | \bm{\sigma}^{(-1)}, \cdots, \bm{\sigma}^{(-n)} \right) \nonumber \\
  &=& \frac{1}{2^{n-1}(T-S)^n} \left[ s_1\left( \bm{\sigma}^{(-1)} \right) - s_2\left( \bm{\sigma}^{(-1)} \right) \right] \prod_{m=2}^{n} \left[ s_2\left( \bm{\sigma}^{(-m)} \right) - s_1\left( \bm{\sigma}^{(-m)} \right) + (T-S) \right] \nonumber \\
 &=& \left[ - \delta_{\bm{\sigma}^{(-1)}, (1,2)} + \delta_{\bm{\sigma}^{(-1)}, (2,1)} \right] \prod_{m=2}^{n} \left[ \frac{1}{2} \delta_{\bm{\sigma}^{(-m)}, (1,1)} + \delta_{\bm{\sigma}^{(-m)}, (1,2)} + \frac{1}{2} \delta_{\bm{\sigma}^{(-m)}, (2,2)} \right]
  \label{eq:TFTn}
\end{eqnarray}
is a memory-$n$ ZD strategy, which unilaterally enforces
\begin{eqnarray}
 0 &=& \left\langle \left\{ s_1\left( \bm{\sigma}^{(-1)} \right) - s_2\left( \bm{\sigma}^{(-1)} \right) \right\} \prod_{m=2}^n \left\{ s_2\left( \bm{\sigma}^{(-m)} \right) - s_1\left( \bm{\sigma}^{(-m)} \right) + (T-S) \right\} \right\rangle^{*}. \nonumber \\
 &&
 \label{eq:linear_TFTn}
\end{eqnarray}
When $n=1$, this strategy is reduced to Tit-for-Tat (TFT) strategy, which unilaterally enforces $\left\langle s_{1} \right\rangle^{*} = \left\langle s_{2} \right\rangle^{*}$ \cite{PreDys2012}.
Because Eq. (\ref{eq:TFTn}) becomes zero when $\bm{\sigma}^{(-m)}=(2,1)$ for some $m\geq 2$, Eq. (\ref{eq:linear_TFTn}) can be regarded as a fairness condition between two players in the situation where $(2,1)$ was not played in the previous $n-1$ rounds.
Since the action profile $(2,1)$ is favorable to player $1$, Eq. (\ref{eq:linear_TFTn}) can be regarded as a fairness condition when player $1$ is in an unfavorable position.
Therefore, this strategy may be stronger than TFT, and whether this statement is true or not should be studied in future.

Similarly, a slightly different memory-$n$ strategy
\begin{eqnarray}
 && \hat{T}_1\left( 1 | \bm{\sigma}^{(-1)}, \cdots, \bm{\sigma}^{(-n)} \right) \nonumber \\
  &=& \frac{1}{2^{n-1}(T-S)^n} \left[ s_1\left( \bm{\sigma}^{(-1)} \right) - s_2\left( \bm{\sigma}^{(-1)} \right) \right] \prod_{m=2}^{n} \left[ s_1\left( \bm{\sigma}^{(-m)} \right) - s_2\left( \bm{\sigma}^{(-m)} \right) + (T-S) \right] \nonumber \\
 &=& \left[ - \delta_{\bm{\sigma}^{(-1)}, (1,2)} + \delta_{\bm{\sigma}^{(-1)}, (2,1)} \right] \prod_{m=2}^{n} \left[ \frac{1}{2} \delta_{\bm{\sigma}^{(-m)}, (1,1)} + \delta_{\bm{\sigma}^{(-m)}, (2,1)} + \frac{1}{2} \delta_{\bm{\sigma}^{(-m)}, (2,2)} \right]
  \label{eq:TFTn2}
\end{eqnarray}
is also a memory-$n$ ZD strategy, which unilaterally enforces
\begin{eqnarray}
 0 &=& \left\langle \left\{ s_1\left( \bm{\sigma}^{(-1)} \right) - s_2\left( \bm{\sigma}^{(-1)} \right) \right\} \prod_{m=2}^n \left\{ s_1\left( \bm{\sigma}^{(-m)} \right) - s_2\left( \bm{\sigma}^{(-m)} \right) + (T-S) \right\} \right\rangle^{*}. \nonumber \\
 &&
 \label{eq:linear_TFTn2}
\end{eqnarray}
This strategy is also reduced to TFT when $n=1$.
Because expectation in Eq. (\ref{eq:linear_TFTn2}) is conditional on the action profiles $(1,1)$, $(2,1)$ and $(2,2)$, this strategy may be weaker than TFT.
We can also construct a mixed version of (\ref{eq:TFTn}) and (\ref{eq:TFTn2}) by choosing different $G_m$ for each $m\geq 2$.

\subsection{Grim Trigger as biased TFT}
\label{subsec:GT}
Finally, we again interpret the Grim Trigger strategy in terms of a biased memory-one ZD strategy.
We remark that the Grim Trigger strategy itself is not a ZD strategy.
In Eq. (\ref{eq:mnZDS_ele}), when we set $a=1$, $n=1$, $\hat{\bm{\sigma}}^{(-1)}=(1,2)$, and choose TFT as the strategy $\hat{T}_1^{(1)}$ of player $1$
\begin{eqnarray}
 \hat{T}_1^{(1)} \left( 1 | \bm{\sigma}^{(-1)} \right) &=& - \delta_{\bm{\sigma}^{(-1)}, (1,2)} + \delta_{\bm{\sigma}^{(-1)}, (2,1)},
\end{eqnarray}
we obtain
\begin{eqnarray}
 \hat{T}_1 \left( 1 | \bm{\sigma}^{(-1)} \right) &=& - \delta_{\bm{\sigma}^{(-1)}, (1,2)},
 \label{eq:PD_GT}
\end{eqnarray}
or $T_1 \left( 1 | \bm{\sigma}^{(-1)} \right) = \delta_{\bm{\sigma}^{(-1)}, (1,1)}$, which is the Grim Trigger strategy.
It should be noted that the action profile $(1,2)$ is unfavorable for player $1$.
Corollary \ref{cor:eleZDS} claims that the strategy unilaterally enforces
\begin{eqnarray}
 P^{*} \left( 1,2 \right) &=& 0,
\end{eqnarray}
which is consistent with Appendix \ref{app:memory-one}.
It has been known that, although a pair of TFT is not a subgame perfect equilibrium, a pair of Grim Trigger is a subgame perfect equilibrium.
Therefore, this result can be interpreted as that TFT gets strengthened by a bias function $\delta_{\bm{\sigma}^{(-1)}, (1,2)}$.

\section{Memory-$n$ deformed ZD strategies}
\label{sec:mnDZDS}
We can consider a memory-$n$ version of deformed ZD strategies (Definition \ref{def:deformedZDS}).
\begin{definition}
\label{def:deformed_mnZDS}
A memory-$n$ strategy of player $a$ is a \emph{deformed memory-$n$ ZD strategy} when its Press-Dyson tensors $\hat{T}_a$ can be written in the form
\begin{eqnarray}
 \hat{T}_a\left( \sigma_a | \bm{\sigma}^{(-1)}, \cdots, \bm{\sigma}^{(-n)} \right) &=& \prod_{m=1}^n G_m \left( \bm{\sigma}^{(-m)} \right),
\end{eqnarray}
where $G_1$ is a deformed memory-one ZD strategy and $G_m$ $(m\geq 2)$ are non-negative quantities.
\end{definition}
For example, in the prisoner's dilemma game, because of the equality
\begin{eqnarray}
 \frac{1}{T-S} \left[ s_1 \left( \bm{\sigma} \right) - s_2 \left( \bm{\sigma} \right) \right] &=& \frac{1}{T^k-S^k} \left[ s_1 \left( \bm{\sigma} \right)^k - s_2 \left( \bm{\sigma} \right)^k \right] \quad (k\geq 1) \\
 &=& \frac{1}{e^{hT}-e^{hS}} \left[ e^{hs_1 \left( \bm{\sigma} \right)} - e^{hs_2 \left( \bm{\sigma} \right)} \right] \quad (\forall h\in \mathbb{R})
\end{eqnarray}
for all $\bm{\sigma}$, the following strategy can be regarded as an extension of TFT strategy:
\begin{eqnarray}
 \hat{T}_1\left( 1 | \bm{\sigma}^{(-1)}, \cdots, \bm{\sigma}^{(-n)} \right) &=& \frac{1}{T-S} \left[ s_1 \left( \bm{\sigma}^{(-1)} \right) - s_2 \left( \bm{\sigma}^{(-1)} \right) \right] \prod_{m=2}^n G_m \left( \bm{\sigma}^{(-m)} \right). \nonumber \\
 &&
\end{eqnarray}
A linear relation enforced by the extended TFT is
\begin{eqnarray}
 0 &=& \left\langle \left[ s_1 \left( \bm{\sigma}^{(-1)} \right)^k - s_2 \left( \bm{\sigma}^{(-1)} \right)^k \right] \prod_{m=2}^n G_m \left( \bm{\sigma}^{(-m)} \right) \right\rangle^{*} \quad (k\geq 1) \\
  &=& \left\langle \left[ e^{hs_1 \left( \bm{\sigma}^{(-1)} \right)} - e^{hs_2 \left( \bm{\sigma}^{(-1)} \right)} \right] \prod_{m=2}^n G_m \left( \bm{\sigma}^{(-m)} \right) \right\rangle^{*}.
\end{eqnarray}
This linear relation can be interpreted as a fairness condition between two players.

In particular, because the identity
\begin{eqnarray}
 && \frac{1}{2(T-S)} \left[ s_2 \left( \bm{\sigma} \right) - s_1 \left( \bm{\sigma} \right) + (T-S) \right] \nonumber \\
 &=& \frac{1}{2\left( e^{hT}-e^{hS} \right)} \left[ e^{hs_2 \left( \bm{\sigma} \right)} - e^{hs_1 \left( \bm{\sigma} \right)} + \left( e^{hT}-e^{hS} \right) \right]
\end{eqnarray}
holds for arbitrary $h$, the strategy (\ref{eq:TFTn}) can be rewritten as
\begin{eqnarray}
 && \hat{T}_1\left( 1 | \bm{\sigma}^{(-1)}, \cdots, \bm{\sigma}^{(-n)} \right) \nonumber \\
  &=& \frac{1}{2^{n-1}\prod_{m=1}^n\left( e^{h_m T}-e^{h_m S} \right)} \left[ e^{h_1 s_1 \left( \bm{\sigma}^{(-1)} \right)} - e^{h_1 s_2 \left( \bm{\sigma}^{(-1)} \right)} \right] \nonumber \\
 && \times \prod_{m=2}^{n} \left[ e^{h_m s_2 \left( \bm{\sigma}^{(-m)}  \right)} - e^{h_m s_1 \left( \bm{\sigma}^{(-m)}  \right)} + \left( e^{h_m T}-e^{h_m S} \right) \right]
\end{eqnarray}
with arbitrary $\{ h_m \}$.
Therefore, we obtain the following proposition:

\begin{proposition}
\label{prop:TFTn}
The strategy (\ref{eq:TFTn}) simultaneously enforces linear relations
\begin{eqnarray}
 0 &=& \left\langle \left\{ e^{h_1 s_1 \left( \bm{\sigma}^{(-1)} \right)} - e^{h_1 s_2 \left( \bm{\sigma}^{(-1)} \right)} \right\} \prod_{m=2}^n \left\{ e^{h_m s_2 \left( \bm{\sigma}^{(-m)}  \right)} - e^{h_m s_1 \left( \bm{\sigma}^{(-m)}  \right)} + \left( e^{h_m T}-e^{h_m S} \right) \right\} \right\rangle^{*} \nonumber \\
 &&
\label{eq:linear_TFTn_mod}
\end{eqnarray}
with arbitrary $\{ h_m \}$.
\end{proposition}

Similarly, the strategy (\ref{eq:TFTn2}) can be rewritten as
\begin{eqnarray}
 && \hat{T}_1\left( 1 | \bm{\sigma}^{(-1)}, \cdots, \bm{\sigma}^{(-n)} \right) \nonumber \\
  &=& \frac{1}{2^{n-1}\prod_{m=1}^n\left( e^{h_m T}-e^{h_m S} \right)} \left[ e^{h_1 s_1 \left( \bm{\sigma}^{(-1)} \right)} - e^{h_1 s_2 \left( \bm{\sigma}^{(-1)} \right)} \right] \nonumber \\
 && \times \prod_{m=2}^{n} \left[ e^{h_m s_1 \left( \bm{\sigma}^{(-m)}  \right)} - e^{h_m s_2 \left( \bm{\sigma}^{(-m)}  \right)} + \left( e^{h_m T}-e^{h_m S} \right) \right]
\end{eqnarray}
with arbitrary $\{ h_m \}$, and the following proposition holds:

\begin{proposition}
\label{prop:TFTn}
The strategy (\ref{eq:TFTn2}) simultaneously enforces linear relations
\begin{eqnarray}
 0 &=& \left\langle \left\{ e^{h_1 s_1 \left( \bm{\sigma}^{(-1)} \right)} - e^{h_1 s_2 \left( \bm{\sigma}^{(-1)} \right)} \right\} \prod_{m=2}^n \left\{ e^{h_m s_1 \left( \bm{\sigma}^{(-m)}  \right)} - e^{h_m s_2 \left( \bm{\sigma}^{(-m)}  \right)} + \left( e^{h_m T}-e^{h_m S} \right) \right\} \right\rangle^{*} \nonumber \\
 &&
\label{eq:linear_TFTn2_mod}
\end{eqnarray}
with arbitrary $\{ h_m \}$.
\end{proposition}

It should be noted that the strategies (\ref{eq:TFTn}) and (\ref{eq:TFTn2}) themselves do not depend on parameters $\{ h_m \}$.
Therefore, the limit probability distribution $P^{*} \left( \bm{\sigma}^{(-1)}, \cdots, \bm{\sigma}^{(-n)} \right)$ also does not depend on $\{ h_m \}$.
By differentiating Eq. (\ref{eq:linear_TFTn_mod}) or Eq. (\ref{eq:linear_TFTn2_mod}) with respect to $\{ h_m \}$ arbitrary times, we can obtain an infinite number of payoff relations.

\section{Concluding remarks}
\label{sec:conclusion}
In this paper, we introduced the concept of biased memory-$n$ ZD strategies (\ref{eq:mnZDS_bias}), which unilaterally enforce linear relations between expected payoffs in biased ensembles or probability zero for a set of action profiles.
Biased memory-$n$ ZD strategies can be used to construct the original memory-$n$ ZD strategies in Ref. \cite{Ued2021b}, which unilaterally enforce linear relations between correlation functions of payoffs.
From another point of view, biased memory-$n$ ZD strategies can be regarded as extension of probability-controlling strategies introduced in Section \ref{sec:trigger}.
Furthermore, biased memory-$n$ ZD strategies can also be used to construct ZD strategies which unilaterally enforce linear relations between conditional expectations of payoffs.
Because the expectation of payoffs conditional on previous action profiles is used, players can choose conditions in such a way that only unfavorable action profiles to one player are contained in the conditions.
We provided several examples of biased memory-$n$ ZD strategies in the repeated prisoner's dilemma game.
Moreover, we explained that extension of deformed ZD strategies \cite{Ued2021} to the memory-$n$ case is straightforward.

The significance of this study is that we provided a method to interpret every time-independent finite-memory strategy in terms of linear relations about $P^*$ enforced by it.
As we can see in Theorem \ref{th:biasZDS}, when a memory-$n$ strategy is described as a biased memory-$n$ ZD strategy, it enforces some linear relation between expected payoffs in biased ensembles or probability zero for a set of action profiles.
Even if a memory-$n$ strategy is not described as a biased memory-$n$ ZD strategy, it can still enforce linear relations about $P^*$ (Lemma \ref{lemma:Akin}).
Such result may be useful for interpretation of strategies.
For example, as we saw in the repeated prisoner's dilemma game, Grim Trigger can be regarded as a memory-one strategy enforcing $P^{*}(1,2)=0$, which directly represents that Grim Trigger is unbeatable.
Furthermore, for any two-player symmetric potential games, the Imitate-If-Better strategy \cite{DOS2012b}, which imitates the opponent's previous action if and only if it was beaten in the previous round, has the similar property as Grim Trigger in the prisoner's dilemma game \cite{Ued2022}.
That is, it unilaterally enforces a linear relation between conditional expectations
\begin{eqnarray}
 0 &=& \left\langle \left\{  s_{a} \left( \bm{\sigma}^{(-1)} \right) - s_{-a} \left( \bm{\sigma}^{(-1)} \right) \right\} \mathbb{I}\left( s_{-a} \left( \bm{\sigma}^{(-1)} \right) > s_{a} \left( \bm{\sigma}^{(-1)} \right) \right) \right\rangle^{*}
\end{eqnarray}
 ($\mathbb{I}$ is an indicator function), which directly means that $\left\langle s_{a} \right\rangle^{*} \geq \left\langle  s_{-a} \right\rangle^{*}$.
(Grim Trigger is a special case of the Imitate-If-Better strategy for the prisoner's dilemma game.)
In this way, controlling conditional expectations can be useful when we investigate properties of finite-memory strategies.

Before ending this paper, we make two remarks.
The first remark is related to the existence of biased memory-$n$ ZD strategies.
It has been known that the existence of memory-one ZD strategies is highly dependent on the stage game \cite{UedTan2020}.
For example, no memory-one ZD strategies exist in the repeated rock-paper-scissors game.
However, because we construct biased memory-$n$ ZD strategies by using memory-one ZD strategies, the existence condition of biased memory-$n$ ZD strategies is the same as that of memory-one ZD strategies used for construction.
Clarifying the existence condition of memory-one ZD strategies is an important subject of future work.

The second remark is on the relation between biased memory-$n$ ZD strategies and equilibrium strategies.
Although we explained that Grim Trigger can be obtained by biasing TFT, other strategies can also be obtained.
For example, in subsection \ref{subsec:GT}, if we use $\hat{\bm{\sigma}}^{(-1)}=(2,1)$ instead of $\hat{\bm{\sigma}}^{(-1)}=(1,2)$, Eq. (\ref{eq:PD_GT}) becomes
\begin{eqnarray}
 \hat{T}_1 \left( 1 | \bm{\sigma}^{(-1)} \right) &=& - \delta_{\bm{\sigma}^{(-1)}, (1,2)},
\end{eqnarray}
which unilaterally enforces $P^{*} \left( 2,1 \right) = 0$.
This strategy forms neither a subgame perfect equilibrium nor a Nash equilibrium.
Therefore, the concept of biased memory-$n$ ZD strategies itself is generally not related to equilibrium strategies.
Rather, an important implication is that all time-independent finite-memory strategies can be interpreted by linear relations about $P^*$ enforced by it, as noted above.
Furthermore, in the repeated prisoner's dilemma game, memory-one Nash equilibria were characterized by using the formalism coming with memory-one ZD strategies \cite{Aki2012}.
In particular, a pair of equalizer strategies, which are one example of memory-one ZD strategies and unilaterally set the payoff of the opponent, forms a Nash equilibrium, because each player cannot improve his/her payoff as long as the opponent uses an equalizer strategy.
We would like to investigate whether our formalism can be used for characterization of memory-$n$ Nash equilibria in the repeated prisoner's dilemma game or not in future.

\appendix

\section{Akin's lemma for deterministic memory-one strategies}
\label{app:memory-one}
In this appendix, we provide results of Akin's lemma in the repeated prisoner's dilemma game.
We use the same notation as that in section \ref{sec:examples}.
We write a memory-one strategy of player $1$ as
\begin{eqnarray}
 \bm{T}_1 (1) &:=&  \left(
\begin{array}{c}
T_1 \left( 1 | 1, 1 \right) \\
T_1 \left( 1 | 1, 2 \right) \\
T_1 \left( 1 | 2, 1 \right) \\
T_1 \left( 1 | 2, 2 \right)
\end{array}
\right).
\end{eqnarray}
The number of deterministic memory-one strategies is sixteen \cite{UsuUed2021}.
The results of Akin's lemma are summarized in the Table \ref{table:Akin}.
\begin{table}[tb]
\caption{The results of Akin's lemma.}
\scalebox{0.85}{
  \begin{tabular}{|c|c|c|} \hline
   strategy $\bm{T}_1 (C)$ & name & Eq. (\ref{eq:gAkin}) \\ \hline
   $(1, 1, 1, 1)^\mathsf{T}$ & All-$C$ & $0 = P^{*}(2,1) + P^{*}(2,2)$ \\
   $(1, 1, 1, 0)^\mathsf{T}$ & & $0 = P^{*}(2,1)$ \\
   $(1, 1, 0, 1)^\mathsf{T}$ & & $0 = P^{*}(2,2)$ \\
   $(1, 1, 0, 0)^\mathsf{T}$ & Repeat & Nothing \\
   $(1, 0, 1, 1)^\mathsf{T}$ & & $0 = - P^{*}(1,2) + P^{*}(2,1) + P^{*}(2,2)$ \\
   $(1, 0, 1, 0)^\mathsf{T}$ & Tit-for-Tat & $0 = - P^{*}(1,2) + P^{*}(2,1)$ \\
   $(1, 0, 0, 1)^\mathsf{T}$ & Win-Stay Lose-Shift & $0 = - P^{*}(1,2) + P^{*}(2,2)$ \\
   $(1, 0, 0, 0)^\mathsf{T}$ & Grim Trigger & $0 = - P^{*}(1,2)$ \\
   $(0, 1, 1, 1)^\mathsf{T}$ & anti-Grim Trigger & $0 = - P^{*}(1,1) + P^{*}(2,1) + P^{*}(2,2)$ \\
   $(0, 1, 1, 0)^\mathsf{T}$ & anti-Win-Stay Lose-Shift & $0 = - P^{*}(1,1) + P^{*}(2,1)$ \\
   $(0, 1, 0, 1)^\mathsf{T}$ & anti-Tit-for-Tat & $0 = - P^{*}(1,1) + P^{*}(2,2)$ \\
   $(0, 1, 0, 0)^\mathsf{T}$ & & $0 = - P^{*}(1,1)$  \\
   $(0, 0, 1, 1)^\mathsf{T}$ & anti-Repeat & $0 = - P^{*}(1,1) - P^{*}(1,2) + P^{*}(2,1) + P^{*}(2,2)$  \\
   $(0, 0, 1, 0)^\mathsf{T}$ & & $0 = - P^{*}(1,1) - P^{*}(1,2) + P^{*}(2,1)$ \\
   $(0, 0, 0, 1)^\mathsf{T}$ & & $0 = - P^{*}(1,1) - P^{*}(1,2) + P^{*}(2,2)$ \\
   $(0, 0, 0, 0)^\mathsf{T}$ & All-$D$ & $0 = - P^{*}(1,1) - P^{*}(1,2)$ \\ \hline
  \end{tabular}
  }
  \label{table:Akin}
\end{table}
We can see that each strategy indeed enforces a linear relation between values of the limit probability distribution $P^{*}$.

\begin{acknowledgements}
We thank Genki Ichinose for valuable discussions.
This study was supported by JSPS KAKENHI Grant Number JP20K19884 and Inamori Research Grants.
\end{acknowledgements}

\section*{Data Availability Statement}
We claim that this work is a theoretical result and there is no available data source.

%
\section*{Conflict of interest}
The author declares that he has no conflict of interest.

\bibliographystyle{spmpsci}      
\bibliography{eleZDS}   

\begin{thebibliography}{10}
\providecommand{\url}[1]{{#1}}
\providecommand{\urlprefix}{URL }
\expandafter\ifx\csname urlstyle\endcsname\relax
  \providecommand{\doi}[1]{DOI~\discretionary{}{}{}#1}\else
  \providecommand{\doi}{DOI~\discretionary{}{}{}\begingroup
  \urlstyle{rm}\Url}\fi

\bibitem{AdaHin2013}
Adami, C., Hintze, A.: Evolutionary instability of zero-determinant strategies
  demonstrates that winning is not everything.
\newblock Nature Communications \textbf{4}(1), 1--8 (2013)

\bibitem{Aki2012}
Akin, E.: The iterated prisoner's dilemma: good strategies and their dynamics.
\newblock Ergodic Theory, Advances in Dynamical Systems pp. 77--107 (2016)

\bibitem{AxeHam1981}
Axelrod, R., Hamilton, W.D.: The evolution of cooperation.
\newblock Science \textbf{211}(4489), 1390--1396 (1981)

\bibitem{BecSch1993}
Beck, C., Sch{\"o}gl, F.: Thermodynamics of Chaotic Systems: an Introduction.
\newblock Cambridge University Press (1993)

\bibitem{DOS2012b}
Duersch, P., Oechssler, J., Schipper, B.C.: Unbeatable imitation.
\newblock Games and Economic Behavior \textbf{76}(1), 88--96 (2012)

\bibitem{Fri1971}
Friedman, J.W.: A non-cooperative equilibrium for supergames.
\newblock The Review of Economic Studies \textbf{38}(1), 1--12 (1971)

\bibitem{FudTir1991}
Fudenberg, D., Tirole, J.: Game Theory.
\newblock MIT Press, Massachusetts (1991)

\bibitem{GJLet2007}
Garrahan, J.P., Jack, R.L., Lecomte, V., Pitard, E., van Duijvendijk, K., van
  Wijland, F.: Dynamical first-order phase transition in kinetically
  constrained models of glasses.
\newblock Physical Review Letters \textbf{98}(19), 195702 (2007)

\bibitem{GKP2006}
Giardina, C., Kurchan, J., Peliti, L.: Direct evaluation of large-deviation
  functions.
\newblock Physical Review Letters \textbf{96}(12), 120603 (2006)

\bibitem{Guo2014}
Guo, J.L.: Zero-determinant strategies in iterated multi-strategy games.
\newblock arXiv preprint arXiv:1409.1786  (2014)

\bibitem{HRZ2015}
Hao, D., Rong, Z., Zhou, T.: Extortion under uncertainty: Zero-determinant
  strategies in noisy games.
\newblock Phys. Rev. E \textbf{91}, 052803 (2015)

\bibitem{HDND2016}
He, X., Dai, H., Ning, P., Dutta, R.: Zero-determinant strategies for
  multi-player multi-action iterated games.
\newblock IEEE Signal Processing Letters \textbf{23}(3), 311--315 (2016)

\bibitem{HMCN2017}
Hilbe, C., Martinez-Vaquero, L.A., Chatterjee, K., Nowak, M.A.: Memory-$n$
  strategies of direct reciprocity.
\newblock Proceedings of the National Academy of Sciences \textbf{114}(18),
  4715--4720 (2017)

\bibitem{HNS2013}
Hilbe, C., Nowak, M.A., Sigmund, K.: Evolution of extortion in iterated
  prisoner{\textquoteright}s dilemma games.
\newblock Proceedings of the National Academy of Sciences \textbf{110}(17),
  6913--6918 (2013)

\bibitem{HNT2013}
Hilbe, C., Nowak, M.A., Traulsen, A.: Adaptive dynamics of extortion and
  compliance.
\newblock PLOS ONE \textbf{8}(11), 1--9 (2013)

\bibitem{HRM2016}
Hilbe, C., R{\"o}hl, T., Milinski, M.: Extortion subdues human players but is
  finally punished in the prisoner's dilemma.
\newblock Nature Communications \textbf{5}, 3976 (2014)

\bibitem{HTS2015}
Hilbe, C., Traulsen, A., Sigmund, K.: Partners or rivals? strategies for the
  iterated prisoner's dilemma.
\newblock Games and Economic Behavior \textbf{92}, 41--52 (2015)

\bibitem{HWTN2014}
Hilbe, C., Wu, B., Traulsen, A., Nowak, M.A.: Cooperation and control in
  multiplayer social dilemmas.
\newblock Proceedings of the National Academy of Sciences \textbf{111}(46),
  16425--16430 (2014)

\bibitem{IchMas2018}
Ichinose, G., Masuda, N.: Zero-determinant strategies in finitely repeated
  games.
\newblock Journal of Theoretical Biology \textbf{438}, 61--77 (2018)

\bibitem{IFN2005}
Imhof, L.A., Fudenberg, D., Nowak, M.A.: Evolutionary cycles of cooperation and
  defection.
\newblock Proceedings of the National Academy of Sciences \textbf{102}(31),
  10797--10800 (2005)

\bibitem{JacSol2010}
Jack, R.L., Sollich, P.: Large deviations and ensembles of trajectories in
  stochastic models.
\newblock Progress of Theoretical Physics Supplement \textbf{184}, 304--317
  (2010)

\bibitem{KalSta1988}
Kalai, E., Stanford, W.: Finite rationality and interpersonal complexity in
  repeated games.
\newblock Econometrica: Journal of the Econometric Society pp. 397--410 (1988)

\bibitem{LAW2005}
Lecomte, V., Appert-Rolland, C., van Wijland, F.: Chaotic properties of systems
  with markov dynamics.
\newblock Physical Review Letters \textbf{95}(1), 010601 (2005)

\bibitem{LiKen2013}
Li, J., Kendall, G.: The effect of memory size on the evolutionary stability of
  strategies in iterated prisoner's dilemma.
\newblock IEEE Transactions on Evolutionary Computation \textbf{18}(6),
  819--826 (2013)

\bibitem{MamIch2019}
Mamiya, A., Ichinose, G.: Strategies that enforce linear payoff relationships
  under observation errors in repeated prisoner's dilemma game.
\newblock Journal of Theoretical Biology \textbf{477}, 63--76 (2019)

\bibitem{MamIch2020}
Mamiya, A., Ichinose, G.: Zero-determinant strategies under observation errors
  in repeated games.
\newblock Phys. Rev. E \textbf{102}, 032115 (2020)

\bibitem{McAHau2016}
McAvoy, A., Hauert, C.: Autocratic strategies for iterated games with arbitrary
  action spaces.
\newblock Proceedings of the National Academy of Sciences \textbf{113}(13),
  3573--3578 (2016)

\bibitem{McAHau2017}
McAvoy, A., Hauert, C.: Autocratic strategies for alternating games.
\newblock Theoretical Population Biology \textbf{113}, 13--22 (2017)

\bibitem{MurBae2018}
Murase, Y., Baek, S.K.: Seven rules to avoid the tragedy of the commons.
\newblock Journal of Theoretical Biology \textbf{449}, 94--102 (2018)

\bibitem{MurBae2020}
Murase, Y., Baek, S.K.: Five rules for friendly rivalry in direct reciprocity.
\newblock Scientific Reports \textbf{10}, 16904 (2020)

\bibitem{Ney1985}
Neyman, A.: Bounded complexity justifies cooperation in the finitely repeated
  prisoners' dilemma.
\newblock Economics Letters \textbf{19}(3), 227--229 (1985)

\bibitem{NeyOka1999}
Neyman, A., Okada, D.: Strategic entropy and complexity in repeated games.
\newblock Games and Economic Behavior \textbf{29}(1-2), 191--223 (1999)

\bibitem{NowSig1993}
Nowak, M., Sigmund, K.: A strategy of win-stay, lose-shift that outperforms
  tit-for-tat in the prisoner's dilemma game.
\newblock Nature \textbf{364}(6432), 56--58 (1993)

\bibitem{NowSig1992}
Nowak, M.A., Sigmund, K.: Tit for tat in heterogeneous populations.
\newblock Nature \textbf{355}(6357), 250--253 (1992)

\bibitem{NyaTou2017}
Nyawo, P.T., Touchette, H.: A minimal model of dynamical phase transition.
\newblock EPL (Europhysics Letters) \textbf{116}(5), 50009 (2017)

\bibitem{OsbRub1994}
Osborne, M.J., Rubinstein, A.: A Course in Game Theory.
\newblock MIT Press, Massachusetts (1994)

\bibitem{PHRT2015}
Pan, L., Hao, D., Rong, Z., Zhou, T.: Zero-determinant strategies in iterated
  public goods game.
\newblock Scientific reports \textbf{5}, 13096 (2015)

\bibitem{PreDys2012}
Press, W.H., Dyson, F.J.: Iterated prisoner{\textquoteright}s dilemma contains
  strategies that dominate any evolutionary opponent.
\newblock Proceedings of the National Academy of Sciences \textbf{109}(26),
  10409--10413 (2012)

\bibitem{RCO1965}
Rapoport, A., Chammah, A.M., Orwant, C.J.: Prisoner's dilemma: A study in
  conflict and cooperation, vol. 165.
\newblock University of Michigan press (1965)

\bibitem{Rub1986}
Rubinstein, A.: Finite automata play the repeated prisoner's dilemma.
\newblock Journal of Economic Theory \textbf{39}(1), 83--96 (1986)

\bibitem{Rub1998}
Rubinstein, A.: Modeling bounded rationality.
\newblock MIT Press, Massachusetts (1998)

\bibitem{StePlo2012}
Stewart, A.J., Plotkin, J.B.: Extortion and cooperation in the
  prisoner{\textquoteright}s dilemma.
\newblock Proceedings of the National Academy of Sciences \textbf{109}(26),
  10134--10135 (2012)

\bibitem{StePlo2013}
Stewart, A.J., Plotkin, J.B.: From extortion to generosity, evolution in the
  iterated prisoner{\textquoteright}s dilemma.
\newblock Proceedings of the National Academy of Sciences \textbf{110}(38),
  15348--15353 (2013)

\bibitem{SzoPer2014}
Szolnoki, A., Perc, M.: Evolution of extortion in structured populations.
\newblock Physical Review E \textbf{89}(2), 022804 (2014)

\bibitem{TraNow2006}
Traulsen, A., Nowak, M.A.: Evolution of cooperation by multilevel selection.
\newblock Proceedings of the National Academy of Sciences \textbf{103}(29),
  10952--10955 (2006)

\bibitem{Ued2021b}
Ueda, M.: Memory-two zero-determinant strategies in repeated games.
\newblock Royal Society Open Science \textbf{8}(5), 202186 (2021)

\bibitem{Ued2021}
Ueda, M.: Tit-for-tat strategy as a deformed zero-determinant strategy in
  repeated games.
\newblock Journal of the Physical Society of Japan \textbf{90}(2), 025002
  (2021)

\bibitem{Ued2022}
Ueda, M.: Unbeatable tit-for-tat as a zero-determinant strategy.
\newblock Journal of the Physical Society of Japan \textbf{91}(5), 054804
  (2022)

\bibitem{UedSas2015}
Ueda, M., Sasa, S.i.: Replica symmetry breaking in trajectories of a driven
  brownian particle.
\newblock Physical Review Letters \textbf{115}(8), 080605 (2015)

\bibitem{UedTan2020}
Ueda, M., Tanaka, T.: Linear algebraic structure of zero-determinant strategies
  in repeated games.
\newblock PLOS ONE \textbf{15}(4), e0230973 (2020)

\bibitem{UsuUed2021}
Usui, Y., Ueda, M.: Symmetric equilibrium of multi-agent reinforcement learning
  in repeated prisoner's dilemma.
\newblock Applied Mathematics and Computation \textbf{409}, 126370 (2021)

\bibitem{WZLZX2016}
Wang, Z., Zhou, Y., Lien, J.W., Zheng, J., Xu, B.: Extortion can outperform
  generosity in the iterated prisoner's dilemma.
\newblock Nature Communications \textbf{7}, 11125 (2016)

\bibitem{YBC2017}
Yi, S.D., Baek, S.K., Choi, J.K.: Combination with anti-tit-for-tat remedies
  problems of tit-for-tat.
\newblock Journal of Theoretical Biology \textbf{412}, 1--7 (2017)

\bibitem{You2017}
Young, R.D.: Press-dyson analysis of asynchronous, sequential prisoner's
  dilemma.
\newblock arXiv preprint arXiv:1712.05048  (2017)

\end{thebibliography}


\begin{thebibliography}{}
%
%
\bibitem{RefJ}
Author, Article title, Journal, Volume, page numbers (year)
\bibitem{RefB}
Author, Book title, page numbers. Publisher, place (year)
\end{thebibliography}

\end{document}